\documentclass[11pt]{amsart}

\usepackage[utf8x]{inputenc}
\usepackage{amsmath}
\usepackage{amssymb}
\usepackage{amsthm}
\usepackage[foot]{amsaddr}
\usepackage[toc,page]{appendix}
\usepackage{enumerate}
\usepackage{fullpage}
\usepackage{graphicx}
\usepackage{lmodern}
\usepackage{mathtools}
\usepackage{multirow}
\usepackage[numbers]{natbib}
\usepackage{remreset}
\usepackage{tabularx}
\usepackage{subcaption}
\usepackage{color}

\usepackage{diagbox}

\theoremstyle{plain}
\newtheorem{theorem}{Theorem}[section]

\newtheorem{cor}{Corollary}[section]
\newtheorem{lemma}{Lemma}[section]
\numberwithin{equation}{section}

\theoremstyle{definition}
\newtheorem{mydef}{Definition}[section]

\newtheorem{rem}{Remark}
\newtheorem{numericalcase}{Case}
\newcommand{\algn}[1]{\begin{align} #1 \end{align}}
\newcommand{\algns}[1]{\begin{align*} #1 \end{align*}}
\newcommand{\algnd}[1]{\begin{aligned} #1 \end{aligned}}
\newcommand{\eqn}[1]{\begin{equation} #1 \end{equation}}
\newcommand{\eqns}[1]{\begin{equation*} #1 \end{equation*}}
\newcommand{\subeqns}[1]{\begin{subequations} #1 \end{subequations}}

\newcommand{\bs}[1]{\boldsymbol{#1}}
\newcommand{\setof}[1]{\left\{ #1 \right\} }	
\newcommand{\sothat}{\,:\,}
\newcommand{\dual}[1]{#1^*}
\newcommand{\adjoint}[1]{#1^*}
\DeclareRobustCommand{\rchi}{{\mathpalette\irchi\relax}}
\newcommand{\irchi}[2]{\raisebox{\depth}{$#1\chi$}} 
\newcommand{\tilrchi}{\tilde{\rchi}}

\renewcommand{\div}{\operatorname{div}}
\newcommand{\curl}{\boldsymbol{\operatorname{curl}}}
\newcommand{\vecdiv}{\boldsymbol{\operatorname{div}}}
\newcommand{\symgrad}[1]{\bs{\epsilon}\left( #1 \right)}
\newcommand{\symgradsymbol}{\bs{\epsilon}}
\newcommand{\skw}{\operatorname{skw}}
\newcommand{\inv}[1]{#1^{-1}}

\newcommand{\grad}{\operatorname{grad}}
\newcommand{\matgrad}{\boldsymbol{\operatorname{grad}}}
\newcommand{\intd}{\mathrm{d}}	

\newcommand{\compl}{\mathcal{C}}			
\newcommand{\pd}{\partial}
\newcommand{\dt}{\Delta t}		
\newcommand{\identity}{I}

\newcommand{\norm}[1]{\left\Vert#1\right\Vert}
\newcommand{\norw}[2]{\left\Vert#1\right\Vert_{#2}}
\newcommand{\inner}[2]{\left( #1, #2\right)}
\newcommand{\trinner}[2]{\left\langle #1, #2 \right\rangle}

\newcommand{\normal}{\hat{\nu}}				
\newcommand{\Vfield}{\mathbb{V}}
\newcommand{\tr}{\operatorname{tr}}
\newcommand{\bkappa}{\boldsymbol{\kappa}}
\newcommand{\bulk}{K}

\newcommand{\idmat}{\mathbb{I}}
\newcommand{\mats}{\mathbb{M}}
\newcommand{\syms}{\mathbb{S}}
\newcommand{\reals}{\mathbb{R}}		
\newcommand{\skews}{\mathbb{K}}		

\newcommand{\A}{A}
\newcommand{\CoeffMat}{\mathcal{A}}
\newcommand{\B}{\mathcal{B}}
\newcommand{\z}{\bs{z}}

\newcommand{\Coperator}{B}

\newcommand{\Bmat}{\mathbb{B}}
\newcommand{\Btmat}{\mathbb{B}_{t}}
\newcommand{\BOmat}{\mathbb{B}_{0}}
\newcommand{\PCmat}{\tilde{\mathbb{B}}}
\newcommand{\PCinv}{\tilde{\mathbb{D}}}
\newcommand{\Vmat}{\mathbb{V}_{\lambda}}
\newcommand{\Dmat}{\mathbb{D}}

\renewcommand{\u}{\bs{u}}
\newcommand{\sig}{\sigma}
\newcommand{\lagmult}{\gamma}
\newcommand{\p}{p}

\renewcommand{\v}{\bs{v}}
\newcommand{\vsig}{\tau}
\newcommand{\vlag}{\eta}
\newcommand{\vp}{q}

\newcommand{\lagspace}{L^2(\Omega;\skews)}

\newcommand{\V}{\bs{V}}
\newcommand{\Q}{Q}

\begin{document}
	\title[]{Weakly imposed symmetry and robust preconditioners  for Biot's consolidation model}\thanks{The research leading to these results has received funding the European Research Council under the European Union's Seventh Framework Programme (FP7/2007-2013) / ERC grant agreement no. 339643. The work of Kent-Andre~Mardal has also been supported by the Research Council of Norway through grant no. 209951 and a Center of Excellence grant awarded to the Center for Biomedical Computing at Simula Research Laboratory.}
	\author{Trygve B\ae rland$^\dagger$}
	\email{trygveba@math.uio.no}
	\address{$^\dagger$Department of Mathematics, University of Oslo, Blindern, Oslo, 0316 Norway}

	\author{Jeonghun J. Lee$^\ddagger$}
	\email{jeonghun@ices.utexas.edu}
	\address{$^\ddagger$The Institute for Computational and Engineering Sciences, University of Texas Austin, 201 East 24th St, Stop C0200
		POB 4.102, Austin, Texas 78712, USA}
	
	\author{Kent-Andre Mardal$^\dagger$}
	\email{kent-and@math.uio.no}
	
	\author{Ragnar Winther$^\dagger$}
	\email{rwinther@math.uio.no}
	
	\begin{abstract} 
			We discuss the construction of robust preconditioners for finite element approximations
			of Biot's consolidation model in poroelasticity. More precisely, we study finite element methods based on 
			generalizations of the Hellinger-Reissner principle of linear elasticity, where the stress tensor is 
			one of the unknowns. The Biot model has a number of applications in science, medicine, and engineering.
			A challenge in many of these applications is that the model parameters range over several orders of magnitude.
			Therefore, discretization procedures which are well behaved with respect to such variations are needed. 
			The focus of the present paper will be on the construction of preconditioners, such that the preconditioned discrete     
			systems are well-conditioned with respect to variations of the model parameters as well as  refinements of 
			the discretization. As a byproduct, we also obtain preconditioners for linear elasticity that are robust in the incompressible limit.
	\end{abstract}
	\maketitle

	\section{Introduction}
	The purpose of this paper is to  discuss a family of finite element methods  for Biot's consolidation model, with a focus on the construction of preconditioners for the discrete systems.
	The Biot model describes the deformation of an elastic porous medium saturated by a viscous fluid, leading to a  system which describes the coupling between the elastic behaviour of the medium and the fluid flow. 
	The finite element systems will therefore contain discrete versions 
	of linear elasticity and porous medium flow as proper subsystems. The methods studied here are based on 
	mixed finite element methods with weakly imposed symmetry for the elasticity part. In this respect, the methods presented 
	here are generalizations of the methods for linear elasticity discussed in \cite{arnold2007mixed}.
	
	With $\Omega$ being an open domain in $\reals^n$, the Biot model is a coupled system of partial differential equations of the form
	\eqn{
		\label{eq:biot_strong_original}
		\algnd{
			-\vecdiv\compl \symgrad{\u} + \alpha \grad \p &= f & &\text{ in } \Omega, \\
			s_0\dot{\p} + \alpha\div \dot{\u} - \div(\bkappa\grad \p) &= g & & \text{ in } \Omega,
		}
	}
	where the dots denote time derivation. The unknowns are the displacement of the structure $\u$, and the pore pressure $p$. The differential operator $\symgradsymbol$ is the symmetric gradient
	and $\compl$ is the stiffness tensor which describes the strain-stress relation. The parameters $s_0$ and $\alpha$ are the so-called constrained specific storage coefficient and the Biot-Willis constant, respectively. 
	Finally, $\bkappa$ is the hydraulic conductivity, determined by the permeability of the medium and the 
	fluid viscosity, while  $f$ and $g$ are given momentum- and mass sources, respectively.
	
	In this paper we will consider linear, isotropic elasticity, in which case the stiffness tensor is modelled as
	\begin{equation}
	\label{isotropic_stress_strain}
	\compl \symgrad{\u} = 2\mu \symgrad{\u} + \lambda\tr\symgrad{\u}\idmat \equiv  2\mu \symgrad{\u} + \lambda(\div\u)\idmat, 
	\end{equation}
	where $\mu, \lambda$ are the Lamé coefficients.
	We will allow the parameters $\mu$, $\lambda$, and $s_0$ to be spatially varying, scalar valued functions, $\bkappa$ is a symmetric positive definite matrix-valued function, while $\alpha \in (0,1]$ is constant. The well-posedness of system (\ref{eq:biot_strong_original}), with appropriate boundary and initial conditions,  is discussed in \citep{showalter2000diffusion}.
	
	The Biot system arises as a key model in many practical applications, such as in geoscience and in the modelling 
	of soft tissues of the central nervous system. For many of these 
	applications the variations of the parameters will be quite large. For example, in geophysical applications  the permeability may vary in the range from $10^{-9}$ to $10^{-21}$ $m^2$, \cite{coussy2004poromechanics,wang2000theory}, while the Lamé coefficient $\lambda$ can vary between $500$ and $10^{6}$ Pa in neurological applications \cite{smith2007interstitial,stoverud2016poro}.
	For a further discussion of relevant properties of the model parameters of the system \eqref{eq:biot_strong_original} we refer to \cite{lee2015parameter} and references  given there. 
	
	Due to the wide range of physical applications of the Biot model there is a need for numerical methods which behave robustly with respect to these  variations of the model parameters.
	A number of finite element methods for the Biot model have previously been proposed in the literature. These studies include various primal methods \cite{reed1984investigation,vermeer1981accuracy,zienkiewicz1984dynamic}, mixed methods \cite{berger2015stabilized,murad1992improved,murad1994stability,yi2014convergence}, and a discontinuous Galerkin method \cite{chen2013analysis}. Combinations of these methods have also been proposed,
	see for example  \cite{lee2016robust,phillips2007coupling,phillips2007coupling2,phillips2008coupling,yi2013coupling},
	while parameter-robust 
	preconditioners are discussed in \cite{axelsson2012stable,haga2012parallel,rhebergen2015threefield}. In fact, this was also the main 
	theme of the paper \cite{lee2015parameter}, where the discretization is based on a standard $H^1$ formulation 
	of the flow, combined with discretizing the elasticity part using stable mixed finite elements for the Stokes equation.
	A standard approach to obtain a locking free displacement method for linear elasticity,
	i.e., a method which behaves well for large Lamé parameters 
	$\lambda$,
	is to introduce ``solid pressure'' as an additional unknown. This approach leads to a three field formulation for 
	the Biot model, where the unknowns are the displacement of the medium and the two pressures.
	The discussion in \cite{lee2015parameter} shows that, in contrast to the situation for linear elasticity,  this approach may not lead to a robust 
	discretization of the Biot system. However, by introducing a new unknown, the so-called ``total pressure'', a robust discretization is obtained. In fact, robustness of the discretization both with respect to the model parameters $\lambda$, 
	$\bkappa$, and the discretization parameter $h$ are obtained. Furthermore,
	robust preconditioners are constructed, i.e., preconditioners that behave uniformly well with respect to 
	variations of the model parameters and refinements of the discretization.
	
	The present paper can be seen as a continuation of \cite{lee2015parameter}, where the discretization of the elasticity part of the system is based on the mixed methods proposed in \cite{arnold2007mixed}. 
	The mixed finite element methods studied in \cite{arnold2007mixed} are based on the Hellinger-Reissner
	variational principle of linear elasticity. An advantage of this approach is that robustness of the methods with respect to the  Lamé parameter $\lambda$ is more or less obtained automatically, and that the 
	stress tensor, which is of more interest in some applications,
	is computed directly. On the other hand, a difficulty of these methods
	is to construct stable finite element function spaces of exactly symmetric stress tensors. Therefore, the 
	methods proposed in \cite{arnold2007mixed}, based on weakly symmetric stresses, are employed.
	In the present paper we generalize these methods to the Biot model. This leads to a four-field formulation 
	where the unknowns are the stress tensor, the displacement of the structure, the pore pressure, and additionally a 
	Lagrange multiplier which results from the weakly imposed symmetry constraint. The main purpose of the present paper 
	is to discuss the properties of these finite element systems. In particular, as in \cite{lee2015parameter}, we will focus on the construction of robust preconditioners for the stationary systems obtained from a time discretization of the evolution problem \eqref{eq:biot_strong_original}.
	
	This
	paper is organized as follows. In Section \ref{sec:preliminaries} we establish the notation that will be used throughout the paper and we give a brief description of the main strategy on how to construct preconditioners that 
	are robust with respect to model parameters and mesh refinement. A proper weak formulation of 
	a semidiscrete version  of the Biot model, with four primary unknowns,  is  also stated  in Section \ref{sec:preliminaries}.
	Section \ref{sec:stability} is devoted to parameter-robust stability results for both the continuous 
	and discrete version of this problem, while more detailed discussions of the construction of the corresponding preconditioners are given in  Section \ref{sec:preconditioning}.
	Finally, in Section \ref{sec:numerical_experiments} we present a few numerical experiments aimed at validating the theoretical results, followed by some concluding remarks in Section \ref{sec:conclusions}.	
	
	\section{Preliminaries}
	\label{sec:preliminaries}
	We will denote by $\Omega$ a bounded domain in $\reals^n$, with $n = 2$ or $3$, and boundary $\pd \Omega$. The space of column $n$-vectors is written $\Vfield = \reals^n$, and $\mats$ will denote the space of $n \times n$ real matrices. Then, $\syms$ and $\skews$ are the subspaces of symmetric- and skew symmetric matrices, respectively.
	
	In the following, $H^k = H^k(\Omega)$ will denote the Sobolev spaces of functions on $\Omega$ with all derivates of order up to $k$ in $L^2(\Omega)$, and its norm is denoted by $\norw{\cdot}{k}$. In addition, $H^k_0$ will denote the closure of $C_0^\infty(\Omega)$ in $H^k$. If $\mathbb{X}$ is an inner product space, $L^2(\Omega;\mathbb{X})$ denotes the space of $\mathbb{X}$-valued, square integrable functions, and  its norm and inner product will be denoted by $\norw{\cdot}{0}$ and $\inner{\cdot}{\cdot}$, respectiely.
	
	Next, $H(\div,\Omega) = H(\div, \Omega; \Vfield)$ will denote the Sobolev space of vector fields on $\Omega$ in $L^2(\Omega; \Vfield)$ with divergence in $L^2(\Omega)$, and its norm is denoted by $\norw{\cdot}{\div} := \left(\norw{\cdot}{0}^2 + \norw{\div\cdot}{0}^2 \right)^{1/2}$. Similarly, $H(\vecdiv, \Omega; \mats)$ will be functions in $L^2(\Omega; \mats)$ with divergence in $L^2(\Omega;\Vfield)$, where the divergence is taken by rows.
	
	For a Hilbert space $X$, we denote its inner product by $\trinner{\cdot}{\cdot}_{X}$, except in the special case of $X = L^2(\Omega)$ already described, in which case $\inner{\cdot}{\cdot}$ is the inner product. If we let $\dual{X}$ denote a representation of the dual of $X$, the duality pairing between $X$ and $\dual{X}$ will be denoted by $\trinner{\cdot}{\cdot}$. We will in the context of Sobolev spaces choose the representation $\dual{X}$ so that the duality pairing is an extension of the $L^2$ inner product. If $Y$ denotes an additional Hilbert space, $\mathcal{L}(X,Y)$ denotes the space of bounded, linear operators from $X$ to $Y$. If $T \in \mathcal{L}(X,\dual{Y})$, we denote its adjoint by $\adjoint{T}$, which is an element of $\mathcal{L}(Y,\dual{X})$.

	\subsection{Abstract preconditioning of parameter dependent systems}
	\label{sec:abstract_preconditioning}
	To motivate the analysis below, we will briefly discuss an abstract framework for preconditioning 
	systems of partial differential equations and their discrete counterparts.
	For a more thorough discussion of this framework we  refer to \citep{lee2015parameter,mardal2011preconditioning}.
	
	Let $X$ be a real, separable Hilbert space.
	Suppose that $\CoeffMat \in \mathcal{L}(X,\dual{X})$ is a linear and bounded operator, which is invertible with bounded inverse. Assume further that $\CoeffMat$ is symmetric, i.e.
	\eqns{
	\trinner{\CoeffMat x}{y} = \trinner{x}{\CoeffMat y}, \quad \forall x,y \in X.	
	}
	We then consider the problem of finding $x\in X$ so that
	\eqn{
		\label{eq:abstract-Hilbert-space-problem}
		\CoeffMat x = f	
	}
	in $\dual{X}$ for a given $f \in \dual{X}$. 
	Applying a symmetric, positive definite operator $\B \in \mathcal{L}(\dual{X},X)$ to problem (\ref{eq:abstract-Hilbert-space-problem}) gives the preconditioned problem of finding $x \in X$ so that
	\eqns{
		\label{eq:abstract-preconditioned-problem}
		\B\CoeffMat x = \B f
	}
	in $X$. The convergence rate of a Krylov subspace method applied to the preconditioned problem is controlled by the condition number
	\eqns{
		K(\B\CoeffMat) = \norw{\B\CoeffMat}{\mathcal{L}(X,X)}\norw{(\B\CoeffMat)^{-1}}{\mathcal{L}(X,X)}
	}
	in the way that a large value of $K(\B\CoeffMat)$ will generally lead to slow convergence.

	We note that one possible choice of the operator $\B$ is the Riesz map from $\dual{X}$ to $X$, or in fact, any operator spectrally equivalent to it. For linear systems arising as discretizations of partial differential equations an effective preconditioner also have to 
	be  easy to evaluate, i.e., we require that the action of the operator can be evaluated cheaply. For systems of partial differential equations this point of view naturally leads to block diagonal preconditioners, where the blocks correspond to preconditioners 
	of  simpler and more canonical operators. For example, 
in the case of operators corresponding to stable discretizations of the inner products of Sobolev spaces like $X= H^1$, $X=H(\curl)$, and $X= H(\div)$,
efficient algorithms that are spectrally equivalent to the Riesz map from $\dual{X}$ to $X$   can be constructed with multilevel algorithms, 
cf. e.g., \citep{arnold1997preconditioning,bramble1993multigrid,hiptmair2007nodal}.
	
	Preconditioning of parameter depedent problems follows in a similar manner. Let $\CoeffMat_\epsilon$ denote an operator depending on some collection of parameters $\epsilon$. To construct a preconditioner for $\CoeffMat_\epsilon$ we determine an $\epsilon$-dependent Hilbert space, $X_\epsilon$, such that $\CoeffMat_\epsilon$ is a linear, symmetric map from $X_\epsilon$ to $\dual{X_\epsilon}$. Furthermore, the corresponding operator norms $\norw{\inv{\CoeffMat_\epsilon}}{\mathcal{L}(\dual{X_\epsilon},X_\epsilon)}$ and $\norw{\CoeffMat_\epsilon}{\mathcal{L}(X_\epsilon, \dual{X_\epsilon})}$ should be bounded independently of $\epsilon$. Having determined $X_\epsilon$, a suitable preconditioner is then a symmetric, positive definite operator $\B_\epsilon$ from $\dual{X_\epsilon}$ to $X_\epsilon$, where the  
	operator norms of $\norw{\B_\epsilon}{\mathcal{L}(\dual{X_\epsilon},X_\epsilon)}$ and  $\norw{\inv{\B_\epsilon}}{\mathcal{L}(X_\epsilon,\dual{X_\epsilon})}$are 
bounded independently of $\epsilon$. We are then guaranteed that the condition number $K(\B_\epsilon\CoeffMat_\epsilon)$ is bounded independently of $\epsilon$, and as a consequence the performance of a Krylov subspace method will 
	basically $\epsilon$ independent.
	
	\subsection{Variational formulation}
	\label{sec:model_formulation}
	
An implicit time discretization of the system (\ref{eq:biot_strong_original}), with time step $\dt$, will typically  lead to a stationary system of the form
	\eqn{
	\label{eq:biot-strong-model-formulation}
		\algnd{
		-\vecdiv\compl \symgrad{\u} + \alpha \grad \p &= f & &\text{ in } \Omega, \\
		s_0\p + \alpha\div \u - \dt\div(\bkappa\grad \p) &= g & & \text{ in } \Omega.
		}
	}
	Here  $g$ encapsulates information about both the mass source and previous time steps. Furthermore, $\dt \bkappa$ can be regarded as one single parameter, which carries information about both the time discretization and the conducivity. Therefore, $\dt$ is set equal to one in the discussion below, while the matrix valued function $\bkappa$ is assumed to be symmetric positive definite, but can be arbitrarily small. For parameter ranges of practical problems, it is typical that $\alpha>0$ is of order 1,
	\algns{ 
	1 \ll \mu \lesssim \lambda \le + \infty,
	}
	and $\mu \ll \lambda$ holds if the elastic matrix is nearly incompressible, i.e. if $\lambda$ is large.
	For the rest of the paper we shall adhere to the following parameter ranges, which are slighty more general than the ranges assumed in \citep{lee2015parameter},
	\eqn{
	\label{eq:parameter-ranges}
	0 < \lambda < +\infty, \quad 0 < \alpha \leq 1, \quad 0 < \bkappa \leq 1.
	}
	Furthermore, the first Lamé coefficient $\mu$ is assumed to be of order $1$. This assumption can be justified by rescaling the equations in \eqref{eq:biot-strong-model-formulation} as well as the parameters $\lambda$, $\alpha$, and $\kappa$ with a constant of order $\mu$, as was done in \citep{lee2015parameter}. In contrast to the discussion presented in \citep{lee2015parameter}, some unknowns ($\sig$ and $\lagmult$, which will be defined below) are also rescaled in this paper. In particular, the variable $\sigma$ is a scaled version of the stress tensor.

	The condition on $\bkappa$ given in (\ref{eq:parameter-ranges}) means that the pointwise eigenvalues of $\bkappa$ are uniformly bounded below by $0$ and above by $1$. 
	The constrained specific storage coefficient is assumed to satisfy the relation $s_0 = \frac{\alpha^2}{\lambda}$. This assumption is mostly for sake of brevity, and the following analysis will work even if $s_0$ is only bounded from below by a constant times $\frac{\alpha^2}{\lambda}$. We refer to \citep{lee2015parameter} for a more detailed discussion of scaling of the Biot system. 
	
	For (\ref{eq:biot-strong-model-formulation}) to be well-posed, it needs to be augmented with a set of boundary conditions. To that end we introduce two separate partitions of the boundary,  $\pd\Omega = \Gamma_\p \cup \Gamma_f = \Gamma_d \cup \Gamma_t$, where $\Gamma_\p$ and $\Gamma_d$ should have positive meaure, i.e., $|\Gamma_p|,|\Gamma_d| > 0$. General boundary conditions can then be posed as
	\eqns{
		\label{eq:boundary_conditions}
		\algnd{
			\p(t) &= \p_0(t) & &\text{ on } \Gamma_\p, \\
			(\bkappa\grad \p(t))\cdot \normal&= \z_{\normal}(t) & &\text{ on } \Gamma_f , \\
			\u(t) &= \u_0(t) & & \text{ on } \Gamma_d , \\
			\sig(t)\normal &= (\compl\symgrad{\u}-\alpha p\idmat)\normal = \sig_{\normal}(t) & &\text{ on } \Gamma_t.
		}
	}
	For simplicity, we will in this paper only consider homogeneous boundary conditions. That is, $\p_0,\z_{\normal}, \u_0, \sig_{\normal} = 0$.
	
	For the weak formulation we introduce a new unknown, the stress tensor, defined as 
	\eqn{
		\label{eq:stress-definition}
		\sig := \compl \symgrad{\u} - \alpha \p\idmat,	
	}
	and we denote the inverse of the stiffness tensor by $\A = \A_{\mu,\lambda} := \inv{\compl}$, which is an operator acting on $\syms$.
	With the stiffness tensor given by (\ref{isotropic_stress_strain}) we obtain
	\begin{equation}
	\label{eq:A-definition}
	\A \sigma = \frac{1}{2\mu}\left(\sigma - \frac{\lambda}{2\mu + n\lambda}\tr(\sigma)\idmat \right).
	\end{equation}
	Furthermore, we note that the trace of (\ref{eq:A-definition}) is given by
	\eqn{
	\label{eq:trace-A}
		\tr\A\sig = \frac{1}{2\mu + n\lambda}\tr\sig.	
	}
	By using (\ref{eq:stress-definition}) and (\ref{eq:trace-A}) we can express the term
	 $\alpha\div\u$ in the second equation of (\ref{eq:biot-strong-model-formulation}) as a function of $\sig$ and $\p$  as
	\algn{
		\label{eq:alpha-divu-relation}
		\alpha \div \u &= \alpha\tr \A (\sig + \alpha\p\idmat) 
		= \bulk\sig + \frac{n\alpha^2}{2\mu+n\lambda}p,
	}
	where $\bulk= \bulk_{\alpha,\mu,\lambda}: \mats \to \reals$ is the operator defined pointwise by
	\eqn{
		\label{eq:bulk-definition}
		K\,\vsig := \frac{\alpha}{2\mu + n\lambda}\tr\vsig.	
	}
	 After introducing $\sig$ defined by (\ref{eq:stress-definition}), and using (\ref{eq:alpha-divu-relation}),
	(\ref{eq:biot-strong-model-formulation}) becomes
	\eqns{
	\algnd{
		\A\sig + \adjoint{\bulk} \p - \symgrad{\u} &= 0 & & \text{ in } \Omega , \\
		\bulk\sig+ \Coperator \p - \div(\bkappa\grad\p) &=g & & \text{ in } \Omega ,\\
		-\vecdiv\sig &= f & & \text{ in } \Omega.
	}
	}
	Here, $\adjoint{\bulk}$ denotes the operator $p \mapsto \frac{\alpha}{2 \mu + n \lambda} p\idmat$, while 
	$\Coperator = \Coperator_{\alpha,\mu,\lambda}$ is the operator defined by
	\eqn{
	\label{eq:Coperator-def}
	\Coperator \p := \left(s_0 + \frac{n\alpha^2}{2\mu+n\lambda}\right)\p \equiv \frac{\alpha^2}{\lambda}\left(1 + \frac{n\lambda}{2\mu +n\lambda}\right)p.	
	} 
	To complete the formulation, we enforce the symmetry of the stress tensor in a weak manner, i.e.,  $\sigma$ is now $\mats$-valued, instead of $\syms$, and we require that
	\begin{equation*}
	\inner{\sig}{\eta} = 0 \quad \forall \eta \in L^2(\Omega; \skews).
	\end{equation*}
	The trade off is that we need to introduce a Lagrange multiplier, $\lagmult$, which will also play the role of the skew symmetric part of $\matgrad \u$. This relaxation of the symmetry on $\sig$ also requires us to extend the definition of $\A$ from $\syms$ to all tensors $\mats$. We denote this extension by $\A$ as well, since it will also be given by formula (\ref{eq:A-definition}).
	
	The system now reads
	\eqn{
		\label{total_system_strong}
		\algnd{
			\A\sig + \adjoint{\bulk} p - \matgrad \u + \lagmult &= 0 & & \text{ in } \Omega, \\
			\bulk\sig + \Coperator \p - \div(\bkappa\grad\p) &=g & & \text{ in } \Omega, \\
			-\vecdiv\sig &= f & & \text{ in } \Omega, \\
			\inner{\sig}{\eta} &= 0& & \forall \eta \in L^2(\Omega; \skews).
		}
	}
	
	Defining the function spaces
	\eqn{
		\label{eq:function_spaces}
		\algnd{
			\Sigma &= \setof{\vsig \in H(\vecdiv,\Omega;\mats) \sothat \vsig\cdot\normal|_{\Gamma_t} = 0}, \\
			\Q &= \setof{\vp \in H^1(\Omega) \sothat \vp|_{\Gamma_p} = 0}, \\
			\V &= L^2(\Omega;\Vfield), \\
			\Gamma &= L^2(\Omega;\skews),
		}
	}
	an appriopriate weak formulation of (\ref{total_system_strong}) is:

	Find $(\sig,\p, \u,\lagmult)\in \Sigma \times \Q \times \V \times \Gamma$ so that
	\eqn{
		\label{hdiv_weaksys}
		\algnd{
			\inner{\A\sig}{\vsig} + \inner{\p}{\bulk\vsig} + \inner{\u}{\vecdiv \vsig} + \inner{\lagmult}{\vsig} &= 0 & & \forall\vsig \in \Sigma, \\
			\inner{\bulk\sig}{\vp } + \inner{\Coperator \p}{\vp} + \inner{\bkappa\grad\p}{\grad\vp} &= \inner{g}{\vp} & &\forall \vp\in \Q, \\
			\inner{\vecdiv \sig}{\v} &= -\inner{f}{\v} & & \forall \v \in \V, \\
			\inner{\sig}{\vlag} &= 0 & & \forall \eta \in \Gamma.
		}
	}
In matrix-vector form, the system (\ref{hdiv_weaksys}) reads
	\eqn{
		\label{eq:coeffmat-tot-system}
		\CoeffMat
		\begin{pmatrix}
			\sig \\
			\p \\
			\u \\
			\lagmult \\
		\end{pmatrix}
		:=
		\begin{pmatrix}
			\A & \adjoint{\bulk} & -\matgrad & \adjoint{\skw} \\
			\bulk & \Coperator-\div(\bkappa\grad) & 0 & 0 \\
			\vecdiv & 0 & 0 & 0 \\
			\skw & 0 & 0 & 0 \\
		\end{pmatrix}
		\begin{pmatrix}
			\sig \\
			\p \\
			\u \\
			\lagmult \\
		\end{pmatrix}
		=
		\begin{pmatrix}
			0 \\
			g \\
			-f \\
			0 \\
		\end{pmatrix},	
	}
	where $\skw: \mats \to \skews$ is the operator returning the skew-symmetric part of a tensor, in which case $\adjoint{\skw}: \skews \to \mats$ is simply the inclusion operator. From (\ref{eq:coeffmat-tot-system}) we see that the system exhibits a saddle point structure, and so well-posedness is ensured if the provided function spaces satisfies the stability conditions in Brezzi's theory of mixed methods (cf. \citep{boffi2013mixed}).
	We introduce the inner products
	\eqn{
		\label{eq:system-innerprods}
		\algnd{
			\trinner{\sig}{\vsig}_{\Sigma} &= \inner{\frac{1}{2\mu}\sig}{\vsig} + \inner{\vecdiv\sig}{\vecdiv \vsig} & & \forall \sig,\vsig \in \Sigma, \\
			\trinner{\p}{\vp}_{\Q} &= \inner{\Coperator \p}{\vp} + \inner{\bkappa\grad\p}{\grad\vp} & & \forall \p,\vp \in \Q, \\
			\trinner{\u}{\v}_{\V} &= \inner{\u}{\v} & & \forall \u,\v \in \V, \\
			\trinner{\lagmult}{\vlag}_{\Gamma} &= \inner{\lagmult}{\vlag} & & \forall \lagmult, \vlag \in \Gamma,
		}	
	}
	and define $\rchi := \Sigma \times \Q \times \V \times \Gamma$ with inner products inherited from (\ref{eq:system-innerprods}).
	With this notation the left hand side of \eqref{hdiv_weaksys} can alternatively be written  as 
	$\trinner{\CoeffMat(\sig,\p, \u, \lagmult)}{(\vsig,\vp, \v, \vlag)}$, where the
	operator $\CoeffMat: \rchi \to \dual{\rchi}$ will be bounded. In fact, in the case when $|\Gamma_t| >0$
	the operator $\CoeffMat$ will be bounded independently 
	of $\alpha$, $\lambda$, and $\bkappa$, and 
	to establish this uniform bound will be a main topic of the next section. However, in the clamped case, i.e., the case when 
	$|\Gamma_t| =0$, we need to alter the norm of the space $\Sigma$ to obtain a corresponding uniform bound.
	This discussion will also be given in the next section.
	
	 We end this section with the following remark. 

	\begin{rem}
		As already noted, the coefficient matrix form in (\ref{eq:coeffmat-tot-system}) exposes the saddle point structure of the system. However, worth noting is that a simple rearrangement of the terms leads to the system
	\eqn{
		\label{eq:coeffmat-alt}
		\CoeffMat
		\begin{pmatrix}
			\sig \\
			\u \\
			\lagmult \\
			\p \\
		\end{pmatrix}
		:=
		\begin{pmatrix}
			\A & -\matgrad & \adjoint{\skw} & \adjoint{\bulk} \\
			\vecdiv & 0 & 0 & 0 \\
			\skw & 0 & 0 & 0 \\
			\bulk & 0 & 0 & \Coperator -\div(\bkappa\grad)\\
		\end{pmatrix}
		\begin{pmatrix}
			\sig \\
			\u \\
			\lagmult \\
			\p \\
		\end{pmatrix}.	
	}
	From this we can consider the system as a coupling between a mixed formulation of linear elasticity with weakly imposed symmetry in the unknown $(\sig, \u, \lagmult)$, and a reaction-diffusion equation in the pore pressure $\p$. We will see that this observation will bear out stable finite element discretizations of this system.
	\end{rem}
	
	\section{Parameter robust stability}
	\label{sec:stability}
	The purpose of this section is to establish stability bounds for the system \eqref{hdiv_weaksys} or equivalently
	\eqref{eq:coeffmat-tot-system}. Note that this system depends on the parameters $\alpha, \mu, \lambda$ implicitly through the definition of the operators $A$, $B$, and $K$, and explicitly of the hydraulic conductivity $\kappa$.
	However, our goal is to establish stability bounds where the stability constant is independent of these parameters, as long as 
	they vary as specified in the beginning of Section \ref{sec:model_formulation}. On the other hand, we will allow the norms to depend on these parameters. More precisely, for the case $|\Gamma_t|>0$ we will use the norms given by the inner products specified in \eqref{eq:system-innerprods}, while the inner product of the space $\Sigma$ has to be altered slightly in the clamped case, i.e., when $|\Gamma_t| = 0$. As we will see in the next section this perturbation will also have an effect on the construction of 
	robust preconditioners.


	\subsection{The continuous case}
	\label{sec:continuous-stability}
	We will first consider the case when $|\Gamma_t| >0$. We introduce the two projections in $L^2(\Omega; \mats)$
	\eqns{
		\label{tensor_projections}
		P_0\tau := \tau - \frac{1}{n}\left(\frac{1}{|\Omega|}\int_{\Omega}\tr \tau \intd x\right) \idmat, \quad P_D\tau := \tau - \frac{1}{n}\tr \tau \idmat.
	}
	That is, $P_0$ projects $\tau$ to its mean trace-free part, whereas $P_D$ projects $\tau$ to its pointwise trace-free part.
	It then follows by algebraic considerations that
	\eqn{
		\label{p0pd_identity}
		P_0P_D = P_DP_0 = P_D. 
	}
	It is worthwhile to note that since $\vecdiv P_0 = \vecdiv$ on $\Sigma$, $P_0$ is also an orthogonal projection on $\Sigma$, not only on $L^2(\Omega;\mats)$.
	An algebraic manipulation gives 
	\eqns{
		\algnd{
		(\A \vsig, \vsig) &= \left( \frac{1}{2 \mu} \vsig, \vsig \right) - \left( \frac{\lambda}{2 \mu (2 \mu + n \lambda)} \tr \vsig, \tr \vsig \right) \\
		& =\left( \frac{1}{2\mu} P_D \vsig, P_D \vsig \right) + \left( \frac{1}{2\mu + n \lambda} (I-P_D) \vsig, (I-P_D) \vsig \right)
		}
	}
	where the second equality follows from $\vsig = P_D \vsig + (I-P_D) \vsig$ and the pointwise orthogonality of $P_D \vsig$ and $(I-P_D) \vsig = \frac{1}{n} \tr \vsig \idmat$. 
	From this a two-side bound of $(\A \vsig, \vsig)$ 
	\eqn{
	\label{A_two_bound}
	\left( \frac{1}{2\mu} P_D \vsig, P_D \vsig \right) \le (\A \vsig, \vsig) \le \left( \frac{1}{2\mu} \vsig, \vsig \right) 
	}  
	follows.
	We will use the following bound 
	\eqn{
		\label{eq:trace-estimate-nonclamped}
		\left( \frac{1}{2\mu} \vsig, \vsig \right)  \leq C\left( \inner{\frac{1}{2\mu}P_D\vsig}{P_D\vsig} + \norw{\vecdiv\vsig}{0}^2\right)	, \quad 
		\vsig \in \Sigma,
	}
	where the constant $C$ is independent of $\vsig$ and $\lambda$. This bound leads to stability of linear elasticity 
	in the incompressible limit, i.e., when $\lambda = +\infty$, and was used already in \cite{arnold1984family}  to obtain robust stability of mixed finite element methods for such problems. This bound will also be crucial for the construction 
	of robust preconditioners for the Biot model, and therefore we will revisit this inequality in the next section. However, as a consequence of the bound \eqref{eq:trace-estimate-nonclamped}, we observe that the following equivalence follows.

	\begin{lemma}
		\label{lem:trace_estimate}
		Assume $\Sigma$ is given by the first definition in (\ref{eq:function_spaces}) with $|\Gamma_t| > 0$. There is a constant  $C > 0$ such  that
		\eqn{
			\label{eq:hdiv-specequiv-nonclamped}
			\inner{\A\vsig}{\vsig} + \norw{\vecdiv\vsig}{0}^2
			\leq \trinner{\vsig}{\vsig}_{\Sigma} \leq C(\inner{\A\vsig}{\vsig} + \norw{\vecdiv\vsig}{0}^2)	
		}
		for every $\vsig \in \Sigma$. In particular, the constant $C$ is independent of $\lambda$.
	\end{lemma}
	\begin{proof}
	The first inequality of \eqref{eq:hdiv-specequiv-nonclamped} follows immediately from \eqref{A_two_bound}.
	The second inequality follows as
	\[
		\trinner{\vsig}{\vsig}_{\Sigma} \leq C\left(			\inner{\frac{1}{2\mu}P_D\vsig}{P_D\vsig} + \norw{\vecdiv\vsig}{0}^2\right) \leq C(\inner{\A\vsig}{\vsig} + \norw{\vecdiv\vsig}{0}^2),
	\] 
	using \eqref{eq:trace-estimate-nonclamped} and \eqref{A_two_bound}.
	\end{proof}

	In the case of $|\Gamma_t| = 0$, the constant matrix field $\vsig = \idmat$ is an element of $\Sigma$, and  $\inner{\A\vsig}{\vsig} + \norw{\vecdiv\vsig}{0}^2 \to  0$ as $\lambda \to + \infty$. 
	Therefore, we cannot hope to extend the $\lambda$-robust equivalence of Lemma \ref{lem:trace_estimate} to the case  $\Gamma_d = \pd \Omega$. In fact, $\vsig = c\idmat$, for any nonzero $c\in\reals$ is the only case that the equivalence fails. Excluding the span of $\{\idmat\}$ from $\Sigma$, we can still have a bound similar to \eqref{eq:hdiv-specequiv-nonclamped} as 
	\eqn{
		\label{eq:trace_estimate_clamped}
		\inner{\frac{1}{2\mu}P_0 \vsig}{P_0 \vsig} \leq C\left( \inner{\frac{1}{2\mu}P_D\vsig}{P_D\vsig} + \norw{\vecdiv\vsig}{0}^2\right),	\quad \vsig \in \Sigma,
	}
which is also proved in \cite{arnold1984family}. 
We can use \eqref{eq:trace_estimate_clamped} to establish that the operator $\A - \matgrad\vecdiv$ is spectrally equivalent to the $\mu$-scaled $H(\vecdiv)$ inner product over $P_0(\Sigma)$, i.e., the subspace of $\Sigma$ consisting of matrix fields with zero mean trace. On the other hand, for $\vsig \in (I-P_0)(\Sigma)$, $\vsig$ is a constant multiple of identity matrix field, so 
	\eqns{
		(\A \vsig, \vsig) = \inner{\frac{1}{2\mu+n\lambda}(I-P_0)\vsig}{(I-P_0)\vsig}.	
	} 
	This gives a motivation to define an auxiliary inner product $\trinner{\cdot}{\cdot}_{\tilde{\Sigma}}$ on $\Sigma$ as
	\eqn{
		\label{eq:auxiliary-stress-innerprod-def}
		\trinner{\sig}{\vsig}_{\tilde{\Sigma}} := \inner{\frac{1}{2\mu}P_0\sig}{P_0\vsig} + \inner{\frac{1}{2\mu+n\lambda}(I-P_0)\sig}{(I-P_0)\vsig} + \inner{\vecdiv\sig}{\vecdiv\vsig}, \quad \sig,\vsig \in \Sigma.
	}
The following lemma states that this inner product is spectrally equivalent to the inner product derived from $\A - \matgrad\vecdiv$.
	\begin{lemma}
		\label{lem:spectral-equivalence-clamped}
		Assume $|\Gamma_t| = 0$. 
		There exists a positive constant $C$ such  that
		\eqn{
			\label{eq:spectral-equivalence-clamped}
			\inv{C}(\inner{\A\vsig}{\vsig} + \norw{\vecdiv\vsig}{0}^2)\leq \trinner{\vsig}{\vsig}_{\tilde{\Sigma}} \leq C (\inner{\A\vsig}{\vsig} + \norw{\vecdiv\vsig}{0}^2).	
		} In particular, the constant $C$ is independent of $\lambda$.
	\end{lemma}
	\begin{proof}
		Since $P_0$ is an orthogonal projection on $\Sigma$, in both inner products, it is sufficient to consider $\vsig \in P_0(\Sigma)$ and $\vsig \in (I-P_0)(\Sigma)$ separately.
		
		If $\vsig \in (I-P_0)(\Sigma)$, then $P_0 \vsig= 0$ and $\vsig$ is a constant multiple of the identity matrix field, so
		\eqns{
			\trinner{\vsig}{\vsig}_{\tilde{\Sigma}} = \inner{\A\vsig}{\vsig} + \norw{\vecdiv\vsig}{0}^2,
		}
		which verifies (\ref{eq:spectral-equivalence-clamped}) in this case.
		
		Next, if $\vsig \in P_0(\Sigma)$, i.e., $\vsig = P_0 \vsig$, we have from \eqref{A_two_bound} and \eqref{eq:trace_estimate_clamped} that
		\eqns{
			\inner{\A\vsig}{\vsig} = \inner{\A P_0\vsig}{P_0\vsig} \le \left(\frac{1}{2\mu} P_0 \vsig, P_0 \vsig \right)  \le C\left(\inner{\frac{1}{2\mu}P_D\vsig}{P_D\vsig} + \norw{\vecdiv\vsig}{0}^2\right)	
		}
		and from the pointwise orthogonality of $(I - P_D)\vsig = (P_0 - P_D) \vsig$ and $P_D \vsig$ that 
		\eqns{
			\trinner{\vsig}{\vsig}_{\tilde{\Sigma}} = \inner{\frac{1}{2\mu}P_0\vsig}{P_0\vsig} + \norw{\vecdiv\vsig}{0}^2 \ge \inner{\frac{1}{2\mu}P_D\vsig}{P_D\vsig} + \norw{\vecdiv\vsig}{0}^2 .	
		}
		The left inequality of \eqref{eq:spectral-equivalence-clamped} easily follows from the above two inequalities. 
		Furthermore, using \eqref{eq:trace_estimate_clamped} and \eqref{A_two_bound}, we obtain 
		\algns{ 
			\trinner{\vsig}{\vsig}_{\tilde{\Sigma}} &= \left( \frac{1}{2\mu} P_0 \vsig, P_0 \vsig \right) + \norw{\vecdiv \vsig}{0}^2 \\
			&\leq C\left( \inner{\frac{1}{2\mu}P_D\vsig}{P_D\vsig} + \norw{\vecdiv\vsig}{0}^2\right) \\
			&\leq C(\inner{\A\vsig}{\vsig} + \norw{\vecdiv\vsig}{0}^2) 
		} 
	which is the right inequality of \eqref{eq:spectral-equivalence-clamped}.

	\end{proof}
	We recall that the space $\rchi = \Sigma \times Q \times \V \times \Gamma$ was introduced in Section \ref{sec:model_formulation} for the case when $|\Gamma_t| > 0$. For the clamped case, i.e., when $|\Gamma_t| = 0$,
	we consider the modified space given by $\tilde{\rchi} := \tilde{\Sigma} \times \Q \times \V \times \Gamma$, where 
	$\tilde{\Sigma} = H(\vecdiv, \Omega; \mats)$, and  with inner product given by (\ref{eq:auxiliary-stress-innerprod-def}). As a consequence of the spectral equivalences \eqref{eq:hdiv-specequiv-nonclamped} and \eqref{eq:spectral-equivalence-clamped} we obtain that the operator $\CoeffMat$ is bounded 
	as an operator in $\mathcal{L}(\rchi, \dual{\rchi})$ when $|\Gamma_t| > 0$, and as an operator in 
	$\mathcal{L}(\tilde{\rchi}, \dual{\tilde{\rchi}})$ in the clamped case when $|\Gamma_t| = 0$.
	More precisely, we have the following result.

		\begin{theorem}
			\label{thm:continuous_stability}
			Assume that the parameters $\lambda$, $\alpha$, and $\bkappa$ satisfies condition 
			(\ref{eq:parameter-ranges}). Let $X = \rchi$ if $|\Gamma_t| >0$, and $X = \tilrchi$ if $|\Gamma_t| = 0$. Then, for the system (\ref{hdiv_weaksys}) there is a constant $\beta > 0$, independent of $\lambda$, $\alpha$, and $\bkappa$, so that the following inf-sup condition holds:
			\eqn{
				\label{eq:infsup-total-system}
				\adjustlimits{\inf}_{(\sig,\p,\u,\lagmult)\in X}  {\sup}_{(\vsig, \vp,\v,\vlag)\in X}
				\frac{\trinner{\CoeffMat(\sig,\p, \u, \lagmult)}{(\vsig,\vp, \v, \vlag)}}{ \norw{(\sig,\p, \u,\lagmult)}{X}  \norw{(\vsig,\vp, \v,\vlag)}{X}}
				\geq
				\beta .
			}
	\end{theorem}
	\begin{proof}
		Consider first the case with $|\Gamma_t | >0$, so that $X = \rchi$.
		To prove (\ref{eq:infsup-total-system}) we will show that there exist positive constants $C_1$, $C_2$, so that for every $0\neq (\sig, \u, \lagmult, \p) \in \rchi$ there are $(\vsig, \v, \vlag, \vp) \in \rchi$ so that
		\eqn{
			\label{eq:infsup-proof-step1}
			\algnd{
				\norw{(\vsig, \vp, \v, \vlag)}{\rchi} &\leq C_1 \norw{(\sig, \p, \u, \lagmult)}{\rchi} , \\
				\trinner{\CoeffMat(\sig,\p, \u,\lagmult)}{(\vsig,\vp,\v,\vlag)} &\geq C_2 \norw{(\sig, \p, \u, \lagmult)}{\rchi}^2,
			}	
		}
		the key being that $C_1$ and $C_2$ will be independent of $\lambda$, $\alpha$, and $\bkappa$.	
			
		To verify (\ref{eq:infsup-proof-step1}), let $(\sig, \u, \lagmult, \p) \in \rchi$ be nonzero, but otherwise arbitrary. From the theory of mixed elasticity with weakly enforced symmetry there exists a $\beta_0>0$, and $\tilde{\vsig}\in \Sigma$ 
		 so that
		\eqn{
			\label{eq:cont-stability-proof-elasticity}
			\algnd{
				\div\tilde{\vsig} &= \u , \\
				\inner{\tilde{\vsig}}{{\vlag}} &= \inner{\lagmult}{{\vlag}} \quad \forall {\vlag} \in \lagspace , \\
				\norw{\tilde{\vsig}}{\Sigma}^2 &\leq \beta_0^2\left(\norw{\u}{0}^2 + \norw{\lagmult}{0}^2 \right),
			} 
		}
		with $\beta_0$ depending only on $\Omega$. From  \eqref{eq:hdiv-specequiv-nonclamped} we see that
		\eqn{ \label{eq:Atau-bound}
			\inner{\A \tilde{\vsig}}{\tilde{\vsig}} \leq \beta_0^2\left(\norw{\u}{0}^2 + \norw{\lagmult}{0}^2 \right).	
		}
By setting $\vsig = \sig + \delta_0\tilde{\vsig}$, $\v = -\u+\delta_1\vecdiv \sig$, $\vlag = -\lagmult$, and $\vp = \p$, we find that
		\eqns{
			\norw{(\vsig,\vp, \v, \vlag)}{\rchi} \leq \sqrt{2(1+\max(\delta_0^2 \beta_0^2,\delta_1^2))}	\norw{(\sig, \p, \u, \lagmult)}{\rchi},	
		}
		which verifies the first inequality in (\ref{eq:infsup-proof-step1}). To prove the second inequality in (\ref{eq:infsup-proof-step1}), we begin by observing that after cancelling terms we obtain the identity
		\algn{
			\label{eq:infsup-proof-step2}
			\trinner{\CoeffMat(\sig,\p, \u,\lagmult)}{(\vsig,\vp, \v,\vlag)} &=
			\inner{\A \sig}{\sig} + \delta_0\inner{\A\sig}{\tilde{\vsig}} + 2\inner{\p}{\bulk \sig} + \norw{p}{Q}^2\\
			&\quad + \delta_0\inner{\p}{\bulk \tilde{\vsig}} + \delta_0\left(\norw{\u}{0}^2 + \norw{\lagmult}{0}^2\right) + \delta_1\norw{\vecdiv \sig}{0}^2 \nonumber,
		}
		where we have used the properties of $\tilde{\vsig}$. To bound the three cross terms we use Cauchy-Schwarz and Young's inequalities in a standard way. For the term $\inner{\A\sig}{\tilde{\vsig}}$, this and \eqref{eq:Atau-bound} yield
		\eqn{
			\label{eq:infsup-proof-step3}
			\inner{\A\sig}{\tilde{\vsig}} \leq \frac{\epsilon_1}{2}\inner{\A\sig}{\sig} + \frac{1}{2\epsilon_1}\inner{\A\tilde{\vsig}}{\tilde{\vsig}} \le \frac{\epsilon_1}{2}\inner{\A\sig}{\sig} + \frac{\beta_0^2}{2\epsilon_1} (\norw{\u}{0}^2 + \norw{\lagmult}{0}^2) 
		}
		for any $\epsilon_1 > 0$.
	We can derive similar bounds for the two terms involving the operator $\bulk$. From the definition of $\bulk$ and Young's inequality we obtain
	\eqn{
		\label{eq:infsup-proof-step4}
		\inner{\p}{\bulk \sig} \leq \frac{\epsilon_2}{2}\inner{\frac{n\alpha^2}{2\mu + n\lambda}\p}{\p} + \frac{1}{2\epsilon_2}\inner{\frac{1}{2\mu + n\lambda}\tr \sig}{\frac{1}{n}\tr \sig}.
	}
	For the first term in \eqref{eq:infsup-proof-step4}, the definition of $\Coperator$ in \eqref{eq:Coperator-def} yields
	\eqns{
		\inner{\frac{n\alpha^2}{2\mu + n\lambda}\p}{\p} \leq \frac{1}{2}\inner{\Coperator \p}{\p} \le \frac{1}{2}\norw{\p}{\Q}^2.
	}
	Inserting this into \eqref{eq:infsup-proof-step4}, and using the properties of $\A$, we obtain
	\algn{
		\label{eq:infsup-proof-step5}
		\inner{\p}{\bulk \sig} &\leq \frac{\epsilon_2}{4}\norw{\p}{Q}^2 + \frac{1}{2\epsilon_2}\inner{\tr \A \sig}{\frac{1}{n}\tr \sig}  \\
		&= \frac{\epsilon_2}{4}\norw{\p}{Q}^2 + \frac{1}{2\epsilon_2}\inner{\A \sig}{\frac{1}{n}\tr \sig \idmat} 
		\leq \frac{\epsilon_2}{4}\norw{\p}{Q}^2 + \frac{1}{2 \epsilon_2}\inner{\A \sig}{\sig}, \nonumber 
	}
		where $\epsilon_2 > 0$ is arbitrary. Furthermore, we have a similar bound 
		\eqns{
		\inner{\p}{\bulk \tilde{\vsig}} \leq \frac{\epsilon_3}{4}\norw{\p}{Q}^2 + \frac{1}{2 \epsilon_3}\inner{\A \tilde{\vsig}}{\tilde{\vsig}} \le \frac{\epsilon_3}{4}\norw{\p}{Q}^2 + \frac{\beta_0^2}{2 \epsilon_3} (\norw{\u}{0}^2 + \norw{\lagmult}{0}^2).
		}
		As a consequence, after using (\ref{eq:infsup-proof-step3}) and (\ref{eq:infsup-proof-step5}) in (\ref{eq:infsup-proof-step2}) and collecting terms, together with using the properties of $\tilde{\vsig}$, we end up with
		\algns{
			\label{eq:infsup-proof-step6}				\trinner{\CoeffMat(\sig, \p, \u, \lagmult)}{(\vsig,\vp,\v,\vlag)} &\geq \left(1 - \frac{\delta_0\epsilon_1}{2} - \frac{1}{\epsilon_2} \right)\inner{\A \sig}{\sig} + \delta_0\left(1 - \frac{\beta_0^2}{2\epsilon_1} - \frac{\beta_0^2}{2\epsilon_3} \right)\left(\norw{\u}{0}^2 + \norw{\lagmult}{0}^2 \right) \nonumber \\
			&\quad + \left( 1 - \frac{\epsilon_2}{2} - \frac{\delta_0\epsilon_3}{4}\right)\norw{\p}{Q}^2 + \delta_1\norw{\vecdiv \sig}{0}^2.
		}
		If we can choose $\delta_0$, $\epsilon_1$, $\epsilon_2$, and $\epsilon_3$ so that all the coefficients above are positive, this will prove the second inequality  in (\ref{eq:infsup-proof-step1}), because of \eqref{hdiv_weaksys}. For instance, choosing
		$\delta_0 = \frac{1}{6\beta_0^2}$, $\epsilon_1 = \epsilon_3 = 2\beta_0^2$, $\epsilon_2 = \frac{3}{2}$, and $\delta_1 = \frac{1}{6}$ yields
		\algns{
			\trinner{\CoeffMat(\sig,\p,\u,\lagmult)}{(\vsig,\vp,\v,\vlag)} &\geq \frac{C}{6}\norw{\sig}{\Sigma} + \frac{1}{12\beta_0^2}\left(\norw{\u}{0}^2 + \norw{\lagmult}{0}^2 \right) + \frac{1}{6}\norw{\p}{Q}^2,
		}
		in which case the second inequality in (\ref{eq:infsup-proof-step1}) holds with $\beta = \frac{1}{6}\min\left(C, \frac{1}{2\beta_0^2} \right)$.
		
		In the case that $X = \tilrchi$ the argument is almost completely analogous. In particular, (\ref{eq:cont-stability-proof-elasticity}) continues to hold with $\norw{\cdot}{\tilde{\Sigma}}$ instead of $\norw{\cdot}{\Sigma}$ since $\norw{\vsig}{\tilde{\Sigma}} \leq \norw{\vsig}{\Sigma}$ for every $\vsig \in \Sigma$.
		When $X = \tilrchi$ we must also use \eqref{eq:spectral-equivalence-clamped} instead of \eqref{eq:hdiv-specequiv-nonclamped}. Other than that, the argument remains unchanged.
	\end{proof}
	
	\subsection{The discrete case	}
	If we discretize (\ref{eq:coeffmat-tot-system}) with finite element spaces $\Sigma_h \subset \Sigma$, $\Q_h \subset \Q$, $\Gamma_h \subset \Gamma$, and $\V_h \subset \V$, and define $\rchi_h = \Sigma_h \times \Q_h \times \V_h \times \Gamma_h$, the discrete formulation becomes:

	Find $(\sig_h,\p_h, \u_h, \lagmult_h) \in \rchi_h$ so that
	\eqn{
		\label{eq:discrete_weaksys}
		\algnd{
			\inner{\A\sig_h}{\vsig} + \inner{\p_h}{\bulk\vsig} + \inner{\u_h}{\vecdiv \vsig} + \inner{\lagmult_h}{\vsig} &=0 & & \forall\vsig \in \Sigma_{h} , \\
			\inner{\bulk\sig_h}{\vp} + \inner{\Coperator\p_h}{\vp} + \inner{\bkappa\grad\p_h}{\grad\vp} &= \inner{g}{\vp} & &\forall \vp\in \Q_{h} ,\\
			\inner{\vecdiv \sig_h}{\v} &= -\inner{f}{\v} & & \forall \v \in \V_h ,\\
			\inner{\sig_h}{\vlag} &= 0 & & \forall \eta \in \Gamma_h. \\
		}
	}
	Alternatively, the left hand side of the system above can be written on the form  $\trinner{\CoeffMat_h(\sig,\p,\u,\lagmult)}{(\vsig,\vp,\v,\vlag)}$,
	where $\CoeffMat_h: \rchi_h \to \dual{\rchi_h}$ is the corresponding discrete coefficient operator. 
	Our goal is to establish a discrete version of Theorem \ref{thm:continuous_stability}, i.e., a stability bound where the stability constant is 
	independent of the model parameters as well as the mesh parameter $h$. We observe that the key feature of the proof of Theorem \ref{thm:continuous_stability} was the property (\ref{eq:cont-stability-proof-elasticity}), which corresponds to the stability of the underlying 
	elasticity problem.  For the proof to carry over to the discrete case, the finite element spaces should satisfy a discrete variant of 
	property (\ref{eq:cont-stability-proof-elasticity}). In other words, the triple $(\Sigma_h, \V_h,\Gamma_h)$ has to be a stable 
	elasticity element. Therefore, we make the following definition.
	
	\begin{mydef}
		\label{def:elasticity_stable}
		We say the function spaces $\Sigma_h$, $\V_h$, and $\Gamma_h$ are elasticity stable if $\vecdiv \Sigma_h = \V_h$, and there exists a constant $C>0$, independent of discretization parameter $h$, such that for any $(\u_h,\lagmult_h) \in \V_h \times \Gamma_h$, there exists $\vsig \in \Sigma_h$ satisfying
		\eqns{
		\label{eq:elasticity-stable}
		\algnd{
			\vecdiv\vsig &= \u_h , \\
			\inner{\vsig}{\vlag} &= \inner{\lagmult_h}{\vlag} \quad \forall \vlag \in \Gamma_h , \\
			\norw{\vsig}{\div} &\leq C\left(\norw{\u_h}{0} + \norw{\lagmult_h}{0} \right).	
		}
		} 
	\end{mydef}
	Examples of elasticity stable elements can be found in \citep{arnold1984peers,arnold2007mixed,boffi2009reduced,cockburn2010new,fortin1997dualhybrid,gopalakrishnan2012second,lee2016unified,stenberg1988elasticity}.
	
	\begin{theorem}
		\label{thm:discrete-stability}
		Let $X = \rchi$ if $\Gamma_t$ has positive measure, and if $|\Gamma_t| = 0$ let $X = \tilrchi$.
		Suppose that $(\Sigma_h, V_h, \Gamma_h)$ in the discrete formulation (\ref{eq:discrete_weaksys}) is elasticity stable, and that the parameter ranges in (\ref{eq:parameter-ranges}) are satisfied. Setting $\rchi_h = \Sigma_h \times \V_h \times \Gamma_h \times Q_h$, with the same norm as $X$, and defining $\CoeffMat_h: \rchi_h \to \dual{\rchi_h}$, then there exists $\beta > 0$ such that
		\eqns{
			\label{eq:infsup-discrete-system}
			\adjustlimits{\inf}_{(\sig,\p,\u,\lagmult)\in \rchi_h}  {\sup}_{(\vsig,\vp,\v,\vlag)\in \rchi_h}
			\frac{\trinner{\CoeffMat_h(\sig,\p,\u,\lagmult)}{(\vsig,\vp,\v,\vlag)}}{ \norw{(\sig,\p,\u,\lagmult)}{X}  \norw{(\vsig,\vp,\v,\vlag)}{X}}
			\geq
			\beta,
		}
		and $\beta$ is independent of $\lambda$, $\alpha$, $\kappa$, and the discretization parameter $h$.
	\end{theorem}
	
	\begin{proof}
		Analogous to the proof of Theorem \ref{thm:continuous_stability}, it is sufficient to prove that there exist constants $C_1$ and $C_2$ so that for every $0 \neq (\sig_h, \p_h, \u_h, \lagmult_h) \in \rchi_h$ there is $(\vsig, \vp, \v,\vlag) \in \rchi_h$ so that
			\algns{
			\norw{(\vsig, \vp, \v, \vlag)}{\rchi} &\leq C_1 \norw{(\sig_h, \p_h, \u_h, \lagmult_h)}{\rchi_h}, \\
			\trinner{\CoeffMat_h(\sig_h,\p_h, \u_h,\lagmult_h)}{(\vsig,\vp,\v,\vlag)} &\geq C_2 \norw{(\sig_h, \p_h, \u_h, \lagmult_h)}{\rchi_h}^2.
		}
		Fix $(\sig_h,\p_h,\u_h,\lagmult_h) \in \rchi_h$.
		Since $\Sigma_h$, $\V_h$, and $\Gamma_h$ are elasticity stable, $\vecdiv\sig_h \in \V_h$ and we can choose $\tilde{\vsig} \in \Sigma_h$ such that
		\algns{
			\vecdiv\tilde{\vsig} &= \u_h , \\
			\inner{\tilde{\vsig}}{\vlag} &= \inner{\lagmult_h}{\vlag} \quad \forall \vlag \in \Gamma_h , \\
			\norw{\tilde{\vsig}}{\div} &\leq C\left(\norw{\u_h}{0} + \norw{\lagmult_h}{0} \right),	
		}
		where the constant $C$ is independent of $h$ and model parameters.
		Setting $\vsig = \sig_h + \delta_0\tilde{\vsig}$, $\vp = p_h$, $\v = -\u_h + \delta_1 \vecdiv\sig_h$, and $\vlag = -\lagmult_h$, we have that $(\vsig,\vp,\v,\vlag) \in \rchi_h$ and
		\algns{
			\label{eq:infsup-discrete-proof-step2}
			\trinner{\CoeffMat(\sig_h,\p_h,\u_h,\lagmult_h)}{(\vsig,\vp, \v,\vlag)} &=
			\inner{\A \sig_h}{\sig_h} + \delta_0\inner{\A\sig_h}{\tilde{\vsig}} + 2\inner{\p_h}{\bulk \sig_h} + \norw{\p_h}{Q}^2\\
			& \quad + \delta_0\inner{\p_h}{\bulk \tilde{\vsig}} + \delta_0\left(\norw{\u_h}{0}^2 + \norw{\lagmult_h}{0}^2\right) + \delta_1\norw{\vecdiv \sig_h}{0}^2 \nonumber.
		}
		The rest of the proof is completely analogous to the proof of Theorem \ref{thm:continuous_stability}.
		
	\end{proof}
	

\section{Preconditioning}
\label{sec:preconditioning}

In this section we will derive order optimal parameter-robust preconditioners for the discretized system. In the case where $|\Gamma_t|>0$
it was shown in the previous section  that the continuous operator $\CoeffMat : \rchi \rightarrow \dual{\rchi}$ was an isomporphism, 
where $\rchi= \Sigma \times \Q \times \V \times \Gamma$.
A parameter-robust preconditioner is then  
constructed as an isomporphism  $\B: \dual{\rchi} \rightarrow \rchi$. The canonical choice, which is
symmetric and positive definite, is:       
\eqn{
\label{eq:block-diagonal-preconditioner1}
\B = \inv{\begin{pmatrix}
\left(\frac{1}{2\mu}  - \matgrad\vecdiv\right) & 0 & 0 & 0 \\
0 & \Coperator - \div\bkappa\grad & 0 & 0 \\
0 & 0 & \identity & 0 \\
0 & 0 & 0 & \identity \\		
\end{pmatrix}}	
}
In the discrete case, order optimal and spectrally equivalent realizations of the preconditioner
can be constructed by multigrid techniques. The first block requires $H(\div)$-preconditioners
such as, e.g., \cite{arnold1997preconditioning,hiptmair2007nodal}. The second block is a second order elliptic 
operator for which multilevel algorithms are well known. If $\V_h$ and $\Gamma_h$ are discontinuous finite element spaces, 
the third and fourth blocks are block diagonal mass matrices and their exact inverses, which are cheaply computable, can be used as preconditioners. 
When $\Gamma_h$ is a Lagrange finite element (e.g., \citep{boffi2009reduced,fortin1997dualhybrid}), 
simple iterative methods such as Jacobi or symmetric Gauss-Seidel give 
preconditioners that are spectrally equivalent to the inverse of the mass matrix.    

The case $|\Gamma_t|=0$ is more challenging and we recall that $\CoeffMat$ is no longer stable in 
$\rchi= \Sigma \times \Q \times \V \times \Gamma$. In fact, stability
was obtained in the alternative space $\tilrchi =\tilde{\Sigma} \times \Q \times \V \times \Gamma$.
Therefore, the canonical choice for a parameter-robust preconditioner is then the symmetric and positive definite operator $\tilde{\B}: \dual{\tilrchi} \rightarrow \tilrchi$ defined by
\eqn{
	\label{eq:block-diagonal-preconditioner2}
	\tilde{\B} = \inv{\begin{pmatrix}
		\left(\frac{1}{2\mu}P_0 + \frac{1}{2\mu+n\lambda}(I-P_0) - \matgrad\vecdiv\right) & 0 & 0 & 0 \\
		0 & \Coperator - \div \bkappa \grad & 0 & 0 \\
		0 & 0 & I & 0 \\
		0 & 0 & 0 & I 
	\end{pmatrix}}.
}
Here, $\tilde{\Sigma}$ is not a function space with standard $H(\div)$ norm,
and it is not clear that the multilevel algorithms developed for standard $H(\div)$ spaces
result in efficient preconditioners in $\tilde{\Sigma}$. 
Here, we will therefore use a technique similar to the one used in \cite{lee2015parameter}.
In the rest of this section we assume that $\mu$, $\lambda$ are constant on $\Omega$.
We recall
the  $\tilde{\Sigma}$ inner product       	
\eqns{
\trinner{\sig}{\vsig}_{\tilde{\Sigma}} = \inner{\frac{1}{2\mu}P_0\sig}{P_0\vsig} + \inner{\frac{1}{2\mu+n\lambda}(I-P_0)\sig}{(I-P_0)\vsig} + \inner{\vecdiv\sig}{\vecdiv\vsig}.
}
To construct a preconditioner for this inner product we rely on the fact that we have efficient preconditioners for the 
weighted $H(\vecdiv,\Omega;\mats)$   inner product 
\[ 
\trinner{\sig}{\vsig}_{\Sigma} = \inner{\frac{1}{2\mu}\sig}{\vsig} + \inner{\vecdiv\sig}{\vecdiv\vsig}.
\]

	Let $\{ \phi_i \}_{i=1}^N$ be a basis for $\Sigma_h \subset \Sigma$. Then we introduce the following matrices: 
	\eqn{
	\label{eq:B-matrices-definitions}
	\algnd{
		\PCmat_{i,j} &= \trinner{\phi_j}{\phi_i}_{\tilde{\Sigma}} , \\
		\Bmat_{i,j} &= \trinner{\phi_j}{\phi_i}_{\Sigma} ,\\
		(\BOmat)_{i,j}&= \frac{1}{2\mu}\inner{P_0\phi_j}{P_0\phi_i} + \inner{\vecdiv\phi_j}{\vecdiv\phi_i} ,\\
		(\Btmat)_{i,j} &= \frac{1}{2\mu}\inner{(I-P_0)\phi_j}{(I-P_0)\phi_i}.
	}
	}
	From  \eqref{eq:B-matrices-definitions}, and (\ref{eq:auxiliary-stress-innerprod-def}), we see that
	\eqns{
		\Bmat = \BOmat + \Btmat, \quad \PCmat = \BOmat + \frac{2\mu}{2\mu + n\lambda}\Btmat.
	}
        Hence,
	\eqns{
		\PCmat = \Bmat - \rho \Btmat \qquad \text{ where } \quad \rho = \frac{n\lambda}{2\mu + n\lambda}.
	}
	Considering the entries of $\Btmat$ in more detail we find that
	\algns{
		(\Btmat)_{i,j} &= \frac{1}{2\mu}\frac{1}{n|\Omega|}\left(\int_{\Omega}\tr\phi_j \intd x\right)\left(\int_{\Omega}\tr\phi_i \intd x\right) \\
		&= \frac{1}{2\mu}m m^T,
	}
	where $m\in \reals^N$ is the column vector with entries
	\eqn{
		\label{eq:m-vector-definition}
		m_i = \frac{1}{\sqrt{n|\Omega |}}\int_{\Omega}\tr\phi_i\intd x.
	}
	Thus, we have that
	\eqn{
		\label{eq:PCmat-id1}
		\PCmat = \Bmat - \frac{\rho}{2\mu}m m^T.
	}
	Next, we define $w\in \reals^N$ to be so that
	\eqn{
		\label{eq:w-vector-definition}
		\sum_{i=1}^N w_i\phi_i = \idmat.
	}
	
	\begin{lemma}
	\label{lem:wmB-vector-identities}
		Let $\{\phi_i\}_{i=1}^N$ be a basis for the finite dimensional function space $\Sigma_h \subset \Sigma$, and assume that $\mu$ and $\lambda$ are positive constants.
		With $m \in \reals^N$ defined by (\ref{eq:m-vector-definition}), $w\in \reals^N$ defined by (\ref{eq:w-vector-definition}), and $\Bmat$ the $N \times N$ matrix defined by (\ref{eq:B-matrices-definitions}), the following identities hold:
		\eqn{
			\label{eq:wm-identities}
			\Bmat w = \frac{\sqrt{n|\Omega |}}{2\mu}m, \quad w^T m = \sqrt{n|\Omega |}.
		}
	\end{lemma}
	\begin{proof}
		For the first identity in (\ref{eq:wm-identities}) we use the definition of $m$, $w$ and $\Bmat$ to see that the $i$'th component of $\Bmat w$ is 
		\algns{
			\left(\Bmat w\right)_i = \sum_{j=1}^N\trinner{w_j\phi_j}{\phi_i}_{\Sigma} = \frac{1}{2\mu}\inner{\idmat}{\phi_i} = \frac{\sqrt{n|\Omega|}}{2\mu}m_i.
		}	
		Similary, the second identity of (\ref{eq:wm-identities}) follows by
		\algns{
			w^T m = \frac{1}{\sqrt{n|\Omega|}}\int_{\Omega}\sum_{i=1}^N \tr w_i\phi_i \intd x = \frac{1}{\sqrt{n|\Omega|}}\int_{\Omega}\tr\idmat \intd x = \sqrt{n|\Omega|}.
		}
	\end{proof}
	\begin{cor}
		\label{cor:VBV-identity}
		Under the same assumptions as in Lemma \ref{lem:wmB-vector-identities}, and $\PCmat$ the $N \times N$ matrix defined by (\ref{eq:B-matrices-definitions}), it holds that 
		\eqn{
			\label{eq:VBV-identity}
			\PCmat = \Vmat^T \Bmat \Vmat,		
		}
		where
		\eqn{
			\label{eq:Vmat-definition}
			\Vmat = \idmat + aw m^T,
		}
		with $a = \frac{1}{\sqrt{n|\Omega|}}(-1 + \sqrt{1-\rho})$. Moreover, $\Vmat$ is invertible with inverse given by
		\eqn{
			\label{eq:V_lambda_inverse}
			\inv{\Vmat} = \idmat + bw m^T,
		} 
		where $b = \frac{1}{\sqrt{n|\Omega|}}\cdot \frac{1 - \sqrt{1-\rho} }{\sqrt{1-\rho}}$.
	\end{cor}
	\begin{proof}
		By matrix multiplication and the identities in (\ref{eq:wm-identities}) we get
		\algns{
			\Vmat^T \Bmat \Vmat &= (\Bmat + \frac{a\sqrt{n|\Omega|}}{2\mu}mm^T)(\idmat + awm^T) \\
			&= \Bmat + \frac{1}{2\mu}(2a\sqrt{n|\Omega|}+a^2n|\Omega|)mm^T,	
		}
		and inserting the value of $a$ yields
		\eqns{
			\Vmat^T \Bmat \Vmat = \Bmat - \frac{\rho}{2\mu} mm^T =  \PCmat .
		}
		This proves (\ref{eq:VBV-identity}) and further,  using the second identity in (\ref{eq:wm-identities}) we see that
		\eqns{
			 (\idmat + bw m^T) \Vmat = \idmat + (a + b + ab\sqrt{n|\Omega|})wm^T. 			
		}
		With the given values of $a$ and $b$ the second term vanishes, so (\ref{eq:V_lambda_inverse}) is proved.  
	\end{proof}

	\begin{lemma}
		\label{lem:stress-preconditioner}
		Suppose that $\Dmat$ is a preconditioner for $\Bmat$ with condition number $K(\Dmat\Bmat)$, then
		\eqn{
			\label{eq:sub_corr_precond}
			\PCinv = \inv{\Vmat}\Dmat\Vmat^{-T},
		}
		where $\inv{\Vmat}$ is given by (\ref{eq:V_lambda_inverse}),
		is a preconditioner for $\PCmat$ and $K(\PCinv\PCmat) = K(\Dmat\Bmat)$. 
                In particular, the condition number is independent of $\lambda$. 
	\end{lemma}
	
	\begin{proof}
		It is seen, using (\ref{eq:VBV-identity}), that
		\algns{
			\PCinv \PCmat &= \inv{\Vmat}\Dmat\Vmat^{-T}\Vmat^T \Bmat \Vmat = \inv{\Vmat}\Dmat \Bmat \Vmat.
		}
		We see from this that $\PCinv \PCmat$ and $\Dmat \Bmat$ are similar matrices, and so their eigenvalues coincide.
	\end{proof}

        Hence, the well-known preconditioners for the  weighted $H(\vecdiv, \Omega; \mats)$  inner product
        $\Bmat$ can be reused such that we obtain a preconditioner spectrally 
        equivalent to $\PCmat$. Furthermore,	the preconditioner $\PCinv$ can then be implemented efficiently by applying $\Vmat^{-T}$,  $\Dmat$, and $\inv{\Vmat}$ sequentially. Note that due to the presence of $wm^T$, $\Vmat^{-T}$ and $\inv{\Vmat}$ are both generally dense matrices. Therefore, the action of $wm^T$ on a vector $v \in \reals^N$ should be implemented as $w(m^Tv)$, i.e., the dot product with $m$ and a scaling of $w$.

	Recall that the stability result of Theorem \ref{thm:continuous_stability} hinges on the spectral equivalences of \eqref{eq:hdiv-specequiv-nonclamped} and \eqref{eq:spectral-equivalence-clamped}, which in turn depend on estimating the trace as given in \eqref{eq:trace-estimate-nonclamped} and \eqref{eq:trace_estimate_clamped}, respectively. Since the parameter-robust stability plays a crucial part in establishing a good preconditioner for the system (\ref{eq:discrete_weaksys}), we will, for the convenience of the reader, include proofs of these inequalities here, even if corresponding arguments can be found in \citep{arnold1984family}.

	\begin{proof}[Proof of \eqref{eq:trace-estimate-nonclamped}]
		Fix $\vsig \in \Sigma$ and recall that $|\Gamma_t| >0$ and $\vsig \cdot \nu = 0$ on $\Gamma_t$. By the pointwise decomposition $\vsig = P_D\vsig + (I-P_D)\vsig$, and the fact that $(I-P_D)\vsig = \frac{1}{n}\tr \vsig \idmat$, it suffices to show that
		\eqns{
			\label{eq:trace_estimate_proof_step1}
			\norw{\tr \vsig}{0}^2 \leq C\left( \inner{\frac{1}{2\mu}P_D\vsig}{P_D\vsig} + \norw{\vecdiv\vsig}{0}^2\right)	
		}
		for some constant $C$ independent of $\vsig$.
		To prove this we use a well-known result for the right inverse of the divergence operator: There exists $\phi \in H^1_{\Gamma_d}(\Omega;\Vfield) := \setof{\varphi \in H^1(\Omega;\Vfield) \sothat \varphi|_{\Gamma_d} = 0}$ such that 
		\eqn{
			\label{eq:stokes_infsup}
			\begin{aligned}
				\div\phi = \tr\vsig, \qquad \norw{\phi}{1} \le C \norw{\tr \vsig}{0}
			\end{aligned}
		}
		with $C>0$ independent of $\vsig$, cf. the Appendix.
		We then have that
		\algns{
			\norm{\tr\vsig}_0^2 &= \inner{\tr\vsig}{\div\phi} = \inner{\tr\vsig\idmat}{\matgrad\phi}.
		}
		Since $\tr\vsig\idmat = n(\vsig - P_D\vsig)$, we get 
		\eqns{
			\norm{\tr\vsig}_0^2 = n\inner{\vsig}{\matgrad \phi} - n\inner{P_D\vsig}{\matgrad \phi}  = -n\inner{\vecdiv\vsig}{\phi}-n\inner{P_D\vsig}{\matgrad \phi},
		}
		where the first term of the final form is a result of integration by parts. Next, we may use Cauchy-Schwarz, which results in
		\algns{
			\norm{\tr\vsig}_0^2 &\leq n\left(\norm{\vecdiv\vsig}_0\norm{\phi}_0 + \norm{P_D\vsig}_0\norm{\grad\phi}_0 \right) \\
			&\leq n\left(\norm{\vecdiv\vsig}_0^2 + \norm{P_D\vsig}_0^2 \right)^{\frac{1}{2}}\norm{\phi}_1 \\
			&\leq C\left(\norm{\vecdiv\vsig}_0^2 + \norm{P_D\vsig}_0^2 \right)^{\frac{1}{2}}\norm{\tr\vsig}_0,
		}
		and so the result follows after dividing by $\norm{\tr\tau}_0$.
	\end{proof}	
	When $|\Gamma_t| = 0$, i.e. $\Gamma_d = \pd \Omega$, \eqref{eq:stokes_infsup} can only hold if $\tr\vsig$ has mean value zero. However, with this constraint, we can prove \eqref{eq:trace_estimate_clamped} with almost the same argument as above.
	\begin{proof}[Proof of \eqref{eq:trace_estimate_clamped}]
		Fix any $\vsig \in \Sigma$. From the decomposition $P_0 \vsig = P_D \vsig + (P_0-P_D) \vsig$, it suffices to prove the estimate for $(P_0 - P_D)\vsig$ component. Denoting the mean value of the trace by
		\eqns{
			\overline{\tr\vsig} := \frac{1}{|\Omega|}\int_{\Omega}\tr\vsig\intd x,
		}
		we have that $(I-P_D)P_0\vsig = (P_0 - P_D)\vsig = \frac{1}{n}(\tr\vsig - \overline{\tr\vsig})\idmat$, and so it is sufficient to show that
		\eqns{
			\norw{\tr\vsig - \overline{\tr\vsig}}{0}^2 \leq C\left(\inner{\frac{1}{2\mu}P_D\vsig}{P_D\vsig} + \norw{\vecdiv\vsig}{0}^2\right).
		}
		Since $\tr \vsig - \overline{\tr\vsig}$ is mean-value zero, there exists $\phi \in \setof{\varphi \in H^1(\Omega;\Vfield) \sothat \varphi|_{\partial \Omega} = 0}$ such that 
		\algns{ 
		\div \phi = \tr \vsig - \overline{\tr\vsig}, \qquad \norw{\phi}{1} \le C \norw{\tr \vsig - \overline{\tr\vsig}}{0}
		}
		with $C >0$ independent of $\vsig$, (cf. \cite[Theorem 5.1]{girault1986finite}).
		The rest of the proof is completely analagous to the proof of \eqref{eq:trace-estimate-nonclamped} above.
	\end{proof}

	\section{Numerical results}
	\label{sec:numerical_experiments}
	In this section we present a series of experiments that demonstrate the performance of the proposed preconditioners.  In all of following numerical experiments $\Omega$ is taken to be the unit square $(0,1)^2$ divided in $N\times N$ squares, where each square is divided in two triangles. The parameters $\alpha$, $\bkappa$, $\mu$, and $\lambda$ are all constants throughout the domain, unless stated otherwise. We let $\mu = \frac{1}{2}$ be fixed but vary  $\alpha$, $\bkappa$, and $\lambda$ in the experiments. Specifically, in Case \ref{case:results-stress-inner-product} we will validate the spectral equivalences \eqref{eq:hdiv-specequiv-nonclamped} and \eqref{eq:spectral-equivalence-clamped} for both fully clamped- and nonclamped boundary conditions. Case \ref{case:results-mixed-elasticity} is concerned with a  linear elasticity system with weakly imposed symmetry under fully clamped conditions as this represent the hardest case.  In Case \ref{case:results-full-biot} the full Biot formulation of \eqref{eq:discrete_weaksys} is preconditioned using a preconditioner based on \eqref{eq:block-diagonal-preconditioner2} and as a final numerical experiment we consider in Case \ref{case:results-variable-kappa} system \eqref{eq:discrete_weaksys} with spatially varying $\bkappa$. The tests are conducted using random right-hand sides and initial guesses. Convergence is reached when the square root of the relative preconditioned residual,  
i.e., $\frac{(B r_k, r_k)}{(B r_0, r_0)}$,  where $r_k$ is the residual at the 
$k$-th iteration and $B$ is the preconditioner, is below a given tolerance. 
 
	\begin{numericalcase}
		\label{case:results-stress-inner-product}
		In the first test case we show the performance of the preconditioners for  the weighted $H(\vecdiv,\Omega; \mats)$ 
		 inner product under nonclamped and clamped conditions. That is,  
		for a given right-hand side $f_h$,  we solve the problem: Find  $\sig_h \in \Sigma_h$ such that  
		\eqn{
		\label{eq:validation-stress-inner-products}
			\inner{\A\sig_h}{\vsig} + \inner{\vecdiv\sig_h}{\vecdiv\vsig} = \inner{f_h}{\vsig} \quad \forall \vsig \in \Sigma_h	.
		}
		We use piecewise linear, row-wise Brezzi-Douglas-Marini (BDM) elements, as described in \citep{arnold2007mixed}. The linear system (\ref{eq:validation-stress-inner-products}) is solved using the preconditioned conjugate gradient method where the choice of preconditioner depends on the boundary conditions. In the case of $|\Gamma_t| > 0$, we use a geometric multigrid procedure with a domain decomposition smoother, c.f.~\citep{arnold1997preconditioning}. Subsequently, this preconditioner will be referred to as the AFW preconditioner. When $|\Gamma_t| = 0$, we construct a preconditioner using (\ref{eq:sub_corr_precond}) and the AFW preconditioner for $\Dmat$. 		The results can be viewed in Table \ref{tab:results_stress_inner_products} where we see that 
the number of iterations remains bounded as $N$ and $\lambda$ vary under both clamped and non-clamped boundary conditions. 
	\end{numericalcase}
	\begin{table}[t]
	\centering
	\setlength\tabcolsep{4pt}
	\begin{minipage}{0.48\textwidth}
		\centering
		\begin{tabular}{l | l l l l l }
		\diagbox{$\lambda$}{$N$} & 4& 8& 16& 32& 64\\
		\hline
		$10^{-4}$ & $3$& $2$& $2$& $2$& $2$\\ 
		$10^{-2}$ & $3$& $2$& $2$& $2$& $2$\\ 
		$10^{0}$ & $6$& $6$& $5$& $5$& $4$\\ 
		$10^{2}$ & $13$& $12$& $11$& $9$& $8$\\ 
		$10^{4}$ & $13$& $13$& $11$& $10$& $8$\\ 
		$10^{6}$ & $13$& $12$& $11$& $9$& $8$\\ 
		$10^{8}$ & $12$& $12$& $11$& $10$& $7$\\ 
		$10^{10}$ & $13$& $13$& $12$& $10$& $8$\\ 
		$10^{12}$ & $12$& $13$& $11$& $10$& $8$\\ 
		\end{tabular}
		\subcaption{$|\Gamma_t| > 0$.}
		\label{tab:res_stress_w_bc_pc}
	\end{minipage}
	\hfill
	\begin{minipage}{0.48\textwidth}
		\centering
		\begin{tabular}{l | l l l l l }
		\diagbox{$\lambda$}{$N$} & 4& 8& 16& 32& 64\\
		\hline
		$10^{-4}$ & $3$& $2$& $2$& $2$& $2$\\ 
		$10^{-2}$ & $3$& $3$& $2$& $2$& $2$\\ 
		$10^{0}$ & $6$& $6$& $5$& $5$& $4$\\ 
		$10^{2}$ & $11$& $11$& $10$& $8$& $7$\\ 
		$10^{4}$ & $9$& $10$& $10$& $8$& $7$\\ 
		$10^{6}$ & $9$& $8$& $9$& $8$& $7$\\ 
		$10^{8}$ & $7$& $7$& $7$& $7$& $7$\\ 
		$10^{10}$ & $7$& $7$& $7$& $7$& $6$\\ 
		$10^{12}$ & $7$& $7$& $8$& $8$& $3$\\ 
	\end{tabular}
	\subcaption{$|\Gamma_t| = 0$.}
	\label{tab:res_stress_innerprod_sc_pc}
	\end{minipage}
	\caption{Number of iterations for solving  (\ref{eq:validation-stress-inner-products})
 using preconditioned conjugate gradient method with error tolerance $10^{-9}$.}
	\label{tab:results_stress_inner_products}
	\end{table}
	
	\begin{numericalcase}
		\label{case:results-mixed-elasticity}
		Before testing the preconditioner on the full Biot system, we present some numerical tests on the reduced system of linear elasticity with weakly enforced symmetry. In our notation, this system takes the following form:
		
		For a given $f_h$, find
		$(\sig_h, \u_h, \lagmult_h)\in \Sigma_h \times \V_h \times \Gamma_h$ so that
		\subeqns{
			\label{eq:validation-mixed-elasticity}
			\algn{
				\inner{\A\sig_h}{\vsig} + \inner{\u_h}{\vecdiv\vsig} + \inner{\lagmult_h}{\vsig} &= 0 & & \forall \vsig \in \Sigma_h, \\
				\inner{\vecdiv\sig_h}{\v} &= -\inner{f_h}{\v} & & \forall \v \in \V_h, \\
				\inner{\sig_h}{\vlag} &= 0 & & \forall \vlag \in \Gamma_h .
			}
		}
		For discretization, we can use any of the stable elements for mixed elasticity with weakly enforced symmetry, see e.g., \citep{arnold2007mixed}.
		In particular, in these numerical experiments we use the same piecewise linear BDM elements for $\Sigma_h$ as in Case \ref{case:results-stress-inner-product}, and piecewise constants for $\V_h$ and $\Gamma_h$.
		Additionally, we only consider fully clamped conditions in this case.
		The system (\ref{eq:validation-mixed-elasticity}) is stable in the inner products in (\ref{eq:auxiliary-stress-innerprod-def}) for $\Sigma_h$, $\V_h$, and $\Gamma_h$, respectively. For preconditioning of the $\Sigma_h$-block we again use (\ref{eq:sub_corr_precond}) together with the AFW preconditioner for $\Dmat$, and for the $\V_h$- and $\Gamma_h$ blocks we use the inverse of the diagonal elements of the corresponding mass matrices. 
		The numerical results can be seen in Table \ref{tab:res_elast_sc_pc}. Here, $N$ denotes the size of the total system. Again the number of iterations remains bounded both as $N$ and $\lambda$ increase.
	\end{numericalcase}
	\begin{table}[t]
		\begin{tabular}{l | l l l l l }
			\diagbox{$\lambda$}{$N$} & 4& 8& 16& 32& 64\\
			\hline
			$10^{-4}$ & $18$& $19$& $19$& $19$& $19$\\ 
			$10^{-2}$ & $18$& $19$& $19$& $19$& $19$\\ 
			$10^{0}$ & $28$& $28$& $28$& $28$& $28$\\ 
			$10^{2}$ & $38$& $41$& $40$& $41$& $42$\\ 
			$10^{4}$ & $35$& $36$& $40$& $41$& $43$\\ 
			$10^{6}$ & $28$& $31$& $36$& $40$& $43$\\ 
			$10^{8}$ & $22$& $24$& $31$& $38$& $38$\\ 
			$10^{10}$ & $20$& $21$& $24$& $35$& $28$\\ 
		\end{tabular}
		\caption{Numerical result for mixed elasticity with weakly enforced symmetry. Table shows number of preconditioned minimal residual iterations until reaching error tolerance $10^{-9}$.}
		\label{tab:res_elast_sc_pc}
	\end{table}
	
	\begin{numericalcase}
		\label{case:results-full-biot}
		Considering the full Biot system with weakly imposed symmetry (\ref{eq:discrete_weaksys}) with fully clamped conditions, we discretize $\Sigma_h$, $\V_h$, and $\Gamma_h$ using the same function spaces as in Case \ref{case:results-mixed-elasticity}, and $\Q_h$ is the space of piecewise continuous linear functions over the triangulation of $\Omega$. The boundary conditions for the pressure are homogeneous Neumann conditions, i.e., $|\Gamma_p| = 0$, and to remove the singularity we fix the value of the pore pressure at a single point. The triple $(\Sigma_h, \V_h, \Gamma_h)$ is elasticity stable, which ensures the stability of Theorem \ref{thm:discrete-stability}, and consequently we can use a preconditioner based on (\ref{eq:block-diagonal-preconditioner2}). The actual preconditioner is then constructed using geometrical multigrid with Jacobi smoother replacing the second block of (\ref{eq:block-diagonal-preconditioner2}) for the pore pressure, while the remaining blocks are treated as in Case \ref{case:results-mixed-elasticity}. The results can be seen in Table \ref{tab:biot_test}, where we see that robustness in $N$ and $\lambda$ continue to hold as well as for $\bkappa$ and $\alpha$. 
	\end{numericalcase}
	\begin{table}[t]
		\begin{tabular}{l| l |l || l l l l }
			\multicolumn{3}{c||}{ } &\multicolumn{4}{c}{$N$}\\ 
			$\bkappa$ & $\alpha$ & $\lambda$ & 4& 8& 16& 32\\ 
			\hline 
			\multirow{6}{*}{$10^{0}$} &\multirow{3}{*}{$10^{0}$} &$10^{0}$ & $18$& $22$& $25$& $43$\\ 
			& & $10^{4}$ & $28$& $31$& $33$& $28$\\ 
			& & $10^{8}$ & $28$& $31$& $35$& $22$\\ 
			\cline{3-7}& \multirow{3}{*}{$10^{-4}$} &$10^{0}$ & $21$& $24$& $22$& $27$\\ 
			& & $10^{4}$ & $28$& $31$& $37$& $23$\\ 
			& & $10^{8}$ & $27$& $31$& $30$& $26$\\ 
			\cline{3-7}\cline{2-3} 
			\multirow{6}{*}{$10^{-4}$} &\multirow{3}{*}{$10^{0}$} &$10^{0}$ & $21$& $19$& $18$& $14$\\ 
			& & $10^{4}$ & $27$& $24$& $19$& $16$\\ 
			& & $10^{8}$ & $26$& $24$& $19$& $16$\\ 
			\cline{3-7}& \multirow{3}{*}{$10^{-4}$} &$10^{0}$ & $18$& $16$& $15$& $12$\\ 
			& & $10^{4}$ & $27$& $24$& $19$& $16$\\ 
			& & $10^{8}$ & $27$& $24$& $19$& $16$\\ 
			\cline{3-7}\cline{2-3} 
			\multirow{6}{*}{$10^{-8}$} &\multirow{3}{*}{$10^{0}$} &$10^{0}$ & $21$& $19$& $18$& $14$\\ 
			& & $10^{4}$ & $25$& $24$& $19$& $16$\\ 
			& & $10^{8}$ & $26$& $24$& $19$& $16$\\ 
			\cline{3-7}& \multirow{3}{*}{$10^{-4}$} &$10^{0}$ & $18$& $18$& $15$& $12$\\ 
			& & $10^{4}$ & $25$& $24$& $21$& $16$\\ 
			& & $10^{8}$ & $25$& $24$& $19$& $16$\\ 
			\cline{3-7}\cline{2-3} 
			\hline 
		\end{tabular}
		\caption{Numerical results for preconditioning (\ref{eq:discrete_weaksys}). Table shows number of preconditioned minimal residual iterations until reaching error tolerance $10^{-9}$.}
		\label{tab:biot_test}
	\end{table}
	
	\begin{numericalcase}
		\label{case:results-variable-kappa}
		As the final experiment we again consider (\ref{eq:discrete_weaksys}), but now with hydraulic conductivity $\bkappa = \kappa\idmat$, where $\kappa$ is variable in $\Omega$ and defined by
		\eqn{
		\label{eq:variable-kappa}
			\kappa(x,y) = \begin{cases}
				\kappa,& \textit{ if } y \in (1/4,3/4) \\
				1, & \textit{ otherwise.}
			\end{cases}	
		}
		The results can be seen in Table \ref{tab:var_kappa_test}, where we again see robustness in all parameters.
	\end{numericalcase}

	\begin{table}[t]
		\begin{tabular}{l| l |l || l l l l }
			\multicolumn{3}{c||}{ } &\multicolumn{4}{c}{$N$}\\ 
			$\kappa$ & $\alpha$ & $\lambda$ & 4& 8& 16& 32\\ 
			\hline 
			\multirow{6}{*}{$10^{0}$} &\multirow{3}{*}{$10^{0}$} &$10^{0}$ & $20$& $22$& $30$& $23$\\ 
			& & $10^{4}$ & $26$& $31$& $22$& $29$\\ 
			& & $10^{8}$ & $27$& $31$& $36$& $28$\\ 
			\cline{3-7}& \multirow{3}{*}{$10^{-4}$} &$10^{0}$ & $21$& $24$& $36$& $27$\\ 
			& & $10^{4}$ & $28$& $31$& $33$& $28$\\ 
			& & $10^{8}$ & $27$& $30$& $35$& $19$\\ 
			\cline{3-7}\cline{2-3} 
			\multirow{6}{*}{$10^{-4}$} &\multirow{3}{*}{$10^{0}$} &$10^{0}$ & $21$& $25$& $30$& $28$\\ 
			& & $10^{4}$ & $27$& $31$& $33$& $26$\\ 
			& & $10^{8}$ & $26$& $31$& $32$& $40$\\ 
			\cline{3-7}& \multirow{3}{*}{$10^{-4}$} &$10^{0}$ & $21$& $23$& $29$& $28$\\ 
			& & $10^{4}$ & $27$& $31$& $34$& $33$\\ 
			& & $10^{8}$ & $27$& $31$& $34$& $34$\\ 
			\cline{3-7}\cline{2-3} 
			\multirow{6}{*}{$10^{-8}$} &\multirow{3}{*}{$10^{0}$} &$10^{0}$ & $22$& $25$& $30$& $26$\\ 
			& & $10^{4}$ & $26$& $31$& $34$& $18$\\ 
			& & $10^{8}$ & $26$& $31$& $34$& $28$\\ 
			\cline{3-7}& \multirow{3}{*}{$10^{-4}$} &$10^{0}$ & $21$& $23$& $29$& $21$\\ 
			& & $10^{4}$ & $26$& $31$& $34$& $42$\\ 
			& & $10^{8}$ & $28$& $31$& $34$& $26$\\ 
			\cline{3-7}\cline{2-3}  
			\hline 
		\end{tabular}
		\caption{Numerical results for system \eqref{eq:discrete_weaksys} with variable $\bkappa$ according to \eqref{eq:variable-kappa} using preconditioner based on \eqref{eq:block-diagonal-preconditioner2}. Table shows number of preconditioned minimal residual iterations until reaching error tolerance $10^{-9}$.}
		\label{tab:var_kappa_test}
	\end{table}
	\section{Conclusions:}
	\label{sec:conclusions}
	We have proposed a new variational formulation of Biot's consolidation model based on stress, displacement, and pressure, where the symmetry of the stress is imposed weakly. The formulation is robustly bounded and stable in a set of parameter-dependent norms. This motivates two preconditioners of the system, depending on the type of boundary conditions considered.
	We also show that the parameter-robust stability continues to hold when the elasticity part is discretized with finite element spaces based on mixed linear elasticity with weakly imposed symmetry, leaving a lot of freedom in the choice of discretization of the pressure.
	
	The theoretical results in this work are backed up by a number of numerical experiments, showing robustness in a wide range of values for the
	shear- and
	bulk elastic moduli, hydraulic conductivity, as well as time- and space discretization parameters.
	\begin{appendix}
	\section{A right inverse of the divergence operator}
	A construction of a right inverse of the divergence operator, as expressed by \eqref{eq:stokes_infsup}, is closely related to the inf-sup condition for the Stokes problem, and therefore well-known. However, we are not aware of a proper reference for the case when $|\pd \Omega| > |\Gamma_t| > 0$, i.e. for the case when $|\Gamma_d| > 0$, but $\Gamma_d$ is not all of $\pd \Omega$. Therefore, for completeness, we include a proof here.
	\begin{lemma}
		\label{lem:right-inverse-divergence-nonclamped}
		Assume $|\Gamma_t| > 0$ and set $H^1_{\Gamma_d}(\Omega; \Vfield) = \setof{\phi \in H^1(\Omega; \Vfield) \sothat \phi |_{\Gamma_d} = 0 }$. Then there is a constant $C > 0$ so that for every $f \in L^2(\Omega)$ there is a $\phi \in H^1_{\Gamma_d}(\Omega; \Vfield)$ so that
		\eqns{
			\div \phi = f, \quad \norw{\phi}{1} \leq C\norw{f}{0}.
		}
	\end{lemma}
	\begin{proof}
		Take any $f \in L^2(\Omega)$. We first decompose $f$ into its mean value zero- and mean value part as $f = f_0 + f_c$ where $f_0 \in L^2_0(\Omega)$ and $f_c = a_f 1_{\Omega}$ for $a_f \in \reals$. Further, we can decompose $H^1_{\Gamma_d}(\Omega; \Vfield) = H^1_{0}(\Omega;\Vfield) \oplus V_1$, where
		\eqns{
			V_1 := \setof{\phi \in H^1_{\Gamma_d}(\Omega; \Vfield) \sothat \inner{\matgrad\phi}{\matgrad\psi} = 0,\, \forall\psi\in H^1_0(\Omega;\Vfield)}.
		}
		Consider then the problem of finding $\zeta \in V_1$ so that
		\eqn{
			\label{eq:appendix-infsup-stokes-step1}
			\inner{\matgrad\zeta}{\matgrad\psi} = \inner{\idmat}{\matgrad\psi},\, \forall\psi\in V_1.	
		}
		By the Lax-Milgram lemma (cf. e.g., \citep[Theorem 4.1.6]{boffi2013mixed}) problem \eqref{eq:appendix-infsup-stokes-step1} has a unique solution $\zeta$ and $\norw{\zeta}{1} \le C_1$ for some constant $C_1>0$ depending on $\Omega$. Taking $\psi = \zeta$ in \eqref{eq:appendix-infsup-stokes-step1} we obtain
		\eqns{
			\int_{\Omega} \div\zeta \intd x = \norw{\matgrad\zeta}{0}^2.	
		}
		Therefore, if we set $\omega = \frac{a_f}{\norw{\matgrad\zeta}{0}^2}\zeta$ we have $\int_{\Omega}\div\omega\intd x = a_f$
		and $\norw{\omega}{1} \le C\norw{f_c}{0}$ for some constant $C$ depending on $\zeta$.
		It follows that $f - \div\omega \in L^2_0(\Omega)$, i.e., $f-\div\omega$ has mean value zero. From the theory of Stokes equation, we can thus find a $\omega_0 \in H^1_0(\Omega;\Vfield)$ so that 
		\eqn{
			\label{eq:appendix-infsup-stokes-step2}
			\div \omega_0 = f - \div\omega,\quad \norw{\omega_0}{1} \leq C_2\norw{f-\div\omega}{0},
		}
		where the constant $C_2$ is independent of $f-\div\omega$ (cf. \cite[Theorem 5.1]{girault1986finite}). We set $\phi = \omega_0 + \omega$, and it follows from \eqref{eq:appendix-infsup-stokes-step2} that $\div\phi = f$. Using the triangle inequality, \eqref{eq:appendix-infsup-stokes-step2} and the properties of $\omega$ we estimate $\norw{\phi}{1}$ as
		\algns{
		\norw{\phi}{1} &\le \norw{\omega_0}{1} + \norw{\omega}{1} \le C(\norw{f-\div\omega}{0} + \norw{f_c}{0}) \le C(\norw{f}{0} + \norw{\omega}{1}) \leq C\norw{f}{0},  	
		}
		which completes the proof.
	\end{proof}
	\end{appendix}
	\newpage
	\clearpage
	\bibliographystyle{abbrv}
	\vspace{.125in}
	\bibliography{biot_preconditioning.bib}

\end{document}